\documentclass[12pt,a4paper,twoside]{article}

\usepackage[centertags]{amsmath}
\usepackage{amsfonts}
\usepackage{amssymb}
\usepackage{amsthm}
\usepackage{ulem}
\usepackage{epsfig}
\usepackage{subfigure}
\usepackage{color}
\usepackage{mathrsfs}
\usepackage[boldsans]{ccfonts}
\usepackage{array}
\usepackage{calc}
\usepackage{titlesec}
\usepackage{graphicx}
\usepackage{enumitem}

\newtheoremstyle{break}
  {\topsep}{\topsep}%
  {\itshape}{}%
  {\bfseries}{}%
  {\newline}{}%
\newtheoremstyle{break1}
  {\topsep}{\topsep}%
  {}{}%
  {\bfseries}{}%
  {\newline}{}%
\newtheoremstyle{break2}
  {\topsep}{\topsep}%
  {}{}%
  {\bfseries}{}%
  {\newline}{}%

\newtheoremstyle{small}
  {\topsep}{\topsep}%
  {\small}{}%
  {\bfseries}{}%
  {\newline}{}%

\theoremstyle{break} \newtheorem{theorem}{Theorem}[section]
\theoremstyle{break} 
\theoremstyle{break1} \newtheorem{remark}[theorem]{Remark}
\theoremstyle{break2} \newtheorem{remarks}[theorem]{Remarks}
\theoremstyle{break} 
\theoremstyle{break1} \newtheorem{definition}[theorem]{Definition} 
\theoremstyle{break} 
\theoremstyle{break} 
\theoremstyle{break} 
\theoremstyle{break} 
\theoremstyle{break1} \newtheorem{problem}[theorem]{Problem}
\theoremstyle{break} 
\theoremstyle{break} 
\theoremstyle{break} 
\theoremstyle{break1} \newtheorem{conclusion}[theorem]{Conclusion}
\theoremstyle{small} \newtheorem{remhist}[theorem]{Remark}
\numberwithin{equation}{section}

\renewcommand{\S}{\mathcal{S}}

\newcommand{\hide}[1]{}
\newcommand{\N}{{\mathbb{N}}}
\newcommand{\R}{{\mathbb{R}}}
\newcommand{\D}{{\mathbb{D}}}
\newcommand{\C}{{\mathbb{C}}}

\newcommand{\B}{{\mathbb{B}}}

\def\Re{\mathop{{\rm Re}}}

\def\HB{\mathop{{\rm Hol}}(\B^n,\C^n)}
\def\Sn{\mathop{\mathbb{S}_n}}

\def\M{\mathcal{M}}
\def\Un{\mathop{\mathbb{U}_n}}
\def\id{\mathop{{\rm id}}}

\begin{document}

\vspace*{-1cm}

\begin{center}
{\Large \bf Is there a Teichm\"uller principle 
in higher dimensions?}
\end{center}

\medskip
\renewcommand{\thefootnote}{\arabic{footnote}}
\begin{center}
{\large Oliver Roth}\\[1mm]
\today\\[2mm]
\end{center}
\medskip

\renewcommand{\thefootnote}{}
\footnotetext{Mathematics Subject Classification (2000) \quad Primary 30C55
  $\cdot$ 32H02
  $\cdot$  49K15}

\medskip

\begin{abstract}
The underlying theme of Teichm\"uller's papers in function theory is a general
principle which asserts that every extremal problem for univalent functions
of one complex variable is connected with an associated  quadratic
differential. The purpose of this paper is to indicate a possible way of 
extending Teichm\"uller's principle to several complex variables. This
approach is based on the Loewner differential equation.

\end{abstract}
\renewcommand{\thefootnote}{\arabic{footnote}}
\section{Introduction}

We denote by $\HB$ the set of 
all holomorphic maps from the open unit ball $\B^n:=\{z \in \C^n \, : \,
||z||<1\}$ equipped with  the standard Euclidean norm $||\cdot||$ of $\C^n$ into $\C^n$. 
Endowed with the compact--open topology of locally uniform convergence, the vector space
$\HB$ becomes a  Fr\'echet space. In geometric function
theory, the \textit{univalent} maps in $\HB$ are of
particular interest and extremal problems 
provide an effective method for
establishing the existence of univalent maps with certain natural properties.

\medskip

This point of view is particularly successful in the classical one--dimensional case, since
 the class
$$ \S:=\{f \in \text{Hol}(\D,\C) \, : \, f(0)=0, f'(0)=1, f
\text{ univalent}
\}$$
of all normalized univalent (or schlicht) functions on the unit disk $\D:=\B^1$
is a compact subset of
$\text{Hol}(\D,\C)$ and so any continuous functional
$J : \S \to \R$ attains its maximum value within the class $\S$. This means
that there exists at least one $F \in \mathcal{S}$ such that $J(f) \le J(F)$
for any $f \in \mathcal{S}$. We call such a map $F$ an \textit{extremal function for
$J$ over $\mathcal{S}$}. 

\medskip

Around 1938, Teichm\"uller \cite{Teichmueller1938} stated a general principle which roughly says that
any extremal problem over $\mathcal{S}$ is associated in a
well--defined way with a quadratic differential. Teichm\"uller did not
go on to give a precise formulation of his principle in upmost generality, 
but he did content himself with applying his principle to a number of specific,
yet characteristic special cases. 
Only much later, Jenkins (see \cite{Jenkins1958}) succeeded in
formulating what is called the General Coefficient Theorem and which can be
regarded as a rigorous version of Teichm\"uller's principle for a fairly 
large class of extremal problems.
\medskip

In order to state Teichm\"uller's principle,
the following notion is useful.

\begin{definition}[\cite{Strebel1984, Pommerenke1975}]
Let $Q$ be a meromorphic function on $\C$. A formal expression of the form
$$ Q(w) \, dw^2$$
is called a \textit{quadratic differential}. A function $F \in \mathcal{S}$ is called
\textit{admissible for the quadratic differential} $Q(w) \, dw^2$, if $F$ maps $\D$
onto $\C$ minus a set of finitely many analytic arcs $w=w(t)$ satisfying $Q(w) dw^2>0$.
\end{definition}

Now, roughly speaking Teichm\"uller's principle asserts that to each quadratic
differential one can associate an extremal problem for univalent functions  in
such a way that the extremal functions are admissible for this quadratic
differential. This principle has a partial converse, first established by
Schiffer \cite{Schiffer1938}, which says that for any extremal problem for univalent
functions one can  associate a quadratic differential $Q(w) \, dw^2$ so that
the corresponding extremal
functions are admissible for $Q(w)\,dw^2$. 

\medskip

For the sake of simplicity, we restrict our discussion to a particular simple,
yet important case and consider for a fixed integer $N \ge 2$ the $N$--th
coefficient functional
$$ J_N(f):=a_N \, , \qquad  f(z)=z+\sum \limits_{k=2}^{\infty} a_k
z^k \in\text{Hol}(\D,\C) \, .$$  
For this functional, one can  state
the Teichm\"uller--Schiffer paradigma in a  precise and simple way, see \cite[Chapter 10.8]{Duren1983}.
We start with the Schiffer differential equation.

\begin{theorem}[Schiffer's Theorem] \label{thm:schiffer}
Let $F(z)=z+\sum \limits_{k=2}^{\infty} A_k z^k \in \mathcal{S}$ 
be an extremal function for  $\Re J_N$ over $\S$ and let
\begin{equation} \label{eq:Pn}
P_N(w):=\sum \limits_{k=1}^{N-1} J_N(F^{k+1}) w^k \, . 
\end{equation}
Then $F$ is admissible for the quadratic differential 
$$-P_N\left(\frac{1}{w}\right) \frac{dw^2}{w^2}\, ,$$
and $F$ is a solution to the differential equation
\begin{equation} \label{eq:schiffer}
 \left[ \frac{z F'(z)}{F(z)} \right]^2 P_N\left(\frac{1}{F(z)}\right)=R_N(z)
  \, ,
\end{equation}
with 
\begin{align} 
R_N(z)&= (N-1) A_n+\sum \limits_{k=1}^{N-1} \left( k A_k z^{k-N}+k
  \overline{A_k} z^{N-k} \right)\, \label{eq:Rn} \\
\intertext{and}
R_N(\kappa) &\ge  0 \quad \text{ for all } \kappa \in \partial \D \text{
  with equality for at least one } \kappa_0 \in \partial \D \,
.\label{eq:schifferboundary} 
\end{align}
\end{theorem}

\begin{remarks} \label{rem:schiffer}
A few of remarks are in order.

\vspace*{-3mm}
\begin{itemize}
\item[(a)] The only functions $F \in \S$ for which Theorem
  \ref{thm:schiffer} actually applies are the Koebe functions $k(z)=z/(1 -\alpha
  z)^2$ with $\alpha^{N-1}=1$. This is a consequence of de Branges' theorem
  \cite{deBranges1985}, that is, the former Bieberbach conjecture, which
  states that $|a_N| \le N$ with equality if and only if $F(z)=z/(1-\eta z)^2$
  with $|\eta|=1$.
However, Schiffer's method works for much more general (``differentiable'')
functionals $ J : \S \to \C$ as  well, see \cite{Duren1983}.
In this  expository paper, we  nevertheless focus mainly on the
simple functional
 $J_N$ for various reasons. First of all, the more general cases of Schiffer's
 theorem are essentially as difficult to prove as the case $J_N$. Second, the
 relation to Teichm\"uller's principle is most easily described for the
 functional $J_N$. Finally, the main issue of this note are extensions to
 higher dimensions with a view toward a higher dimensional Bieberbach
 conjecture (see e.g.~\cite{BracciRoth2017}).

\item[(b)] In order to  prove  Theorem \ref{thm:schiffer} one compares the
  extremal function with nearby functions in the class $S$. The
  construction of  suitable comparison functions is a nontrivial task, because
  the family $\S$ is highly nonlinear. In one dimension there are several variational methods
  for univalent functions available. We mention the work of Schiffer,
  Schaeffer \& Spencer, Goluzin and others, see \cite{Duren1983, Pommerenke1975}. These methods  are
  geometric in nature and make use of the Riemann mapping theorem and are thus
  specifically one dimensional. A different approach is possible by way of the
  Loewner equation and Pontyagin's Maximum Principle from optimal control. We
  explain this in more detail  below.

\item[(c)] 
Equation (\ref{eq:schiffer}) is called the \textit{Schiffer differential equation}. It
is analogous to the Euler equation in the classical calculus of
variations. Like the Euler equation, it expresses the fact that every extremal
function is a critical point of $J_N$. However, an additional difficulty arises, because the 
Schiffer differential equation involves the initial coefficients $A_2, \ldots,
A_{N-1}$ of the unkown extremal function. 

\item[(d)] It is not difficult to show that
 $F \in \S$
satisfies Schiffer's equation (\ref{eq:schiffer}) such that the ``positivity condition''
 (\ref{eq:schifferboundary}) holds if and only if  $F$ is admissible for the quadratic
 differential $-P_N(1/w) \frac{dw^2}{w^2}$. See the proof of Theorem
 \ref{thm:schiffer} in \cite{Duren1983} for the ``only if''--part and
 e.g.~\cite[Proof of Theorem 7.5]{Pommerenke1975} for the ``if''--part.

\item[(e)] Theorem \ref{thm:schiffer} can further be strengthened by showing that
if  $F \in \mathcal{S}$ is
extremal for the functional $J_N$ over $\S$, then $F$ is a one--slit map, that
is, $F$ maps $\D$ onto $\C$ minus a single analytic arc. Moreover,
this arc has several additional
geometric properties such as increasing modulus, the $\pi/4$--property and an
asymptotic direction at infinity, see \cite[Chapter 10]{Duren1983} for more on this.
\end{itemize}
\end{remarks}

\begin{theorem}[Teichm\"uller's Coefficient Theorem \cite{Teichmueller1938}] \label{thm:teich}
Let $P(w)=w^{N-1}+c_{N-2} w^{N-2}+\cdots +c_1
w+c_0$ be a 
polynomial. Suppose that
$$F(z)=z+\sum\limits_{k=2}^{\infty} A_k z^k \in \mathcal{S}$$ is admissible for the
quadratic differential
$$-P\left(\frac{1}{w}\right) \frac{dw^2}{w^2}\, .$$
Then 
$$ \Re J_N(f) \le \Re J_N(F)$$
for any  
$$f \in \mathcal{S}(A_2,\ldots, A_{N-1})
:=\left\{f(z)=z+\sum \limits_{k=2}^{\infty} a_k z^k \in \mathcal{S} \, : \,
a_2=A_2, \ldots, a_{N-1}=A_{N-1}\right\}\, .$$ Equality occurs only for $f=F$. 
\end{theorem}

In short, under the assumptions of Theorem \ref{thm:teich}, that is, if $F \in
\S$ is a solution to the Schiffer differential equation (\ref{eq:schiffer})
such that the positivity condition (\ref{eq:schifferboundary}) holds,
then  $F$  is an extremal function for the real part of the Bieberbach functional $J_N(f)$,
 but subject to the \textit{side conditions}  $a_2=A_2, \ldots,
 a_{N-1}=A_{N-1}$:

\begin{conclusion}
Let $F(z)=z+\sum\limits_{k=2}^{\infty} A_k z^k \in \mathcal{S}$. Then the condition that
\begin{center}  \fboxsep5mm
\fbox{$F$ is a solution to Schiffer's differential equation
  (\ref{eq:schiffer}) such that (\ref{eq:schifferboundary}) holds}
\end{center}
is
\begin{itemize}
\item[(a)] necessary for $F$ being extremal for $\Re J_N$ over the entire
  class $\S$, and
\item[(b)] sufficient for $F$ being extremal for $\Re J_N$ over the restricted
  class $\S(A_2,\ldots, A_{N-1})$.
\end{itemize}
\end{conclusion}

\begin{remark}
Teichm\"uller was quite confident about his result and he conjectured\footnote{``Ich vermute, die Gesamtheit dieser Ungleichungen liefere eine
vollst\"andige L\"osung des Bieberbachschen Koeffizientenproblems'' \cite[p.~363]{Teichmueller1938}} that it
can be used to solve the 

\begin{center}\fboxsep3mm
\fbox{\begin{minipage}{13cm}
\textit{\bf General Coefficient Problem for Univalent Functions}

\medskip
Given $F(z)=z+\sum\limits_{k=2}^{\infty} A_k z^k \in \mathcal{S}$. Find for
each $N \in \N$,
$$ \left\{a_N \, : f(z)=z+\sum \limits_{k=2}^{\infty} a_k z^k\in \S(A_2,\ldots, A_{N-1}) \right\} \, .$$
\end{minipage}}
\end{center}
\end{remark}

\medskip

The goal of this note is to show that Schiffer's theorem can  be extended, at
least in spirit, to higher dimensions using the Loewner equation. We also give a statement of  
Teichm\"uller's Coefficient Theorem entirely in terms of the Loewner equation.
This  principally opens up the possibility for an extension of Teichm\"uller's
principle to higher dimensions. 

\medskip

The literature on univalent functions in general, and the Loewner equation in
particular, is
extensive and there are several excellent survey papers available, see, e.g.~\cite{BCTV2014}.
We therefore have included only few references about the subject.
As this  paper is expository, it contains virtually no proof. An exception is Theorem \ref{thm:teichneu}.

\hide{\section*{Obstructions in higher dimensionss}

\subsection{Noncompactness and lack of Riemann mapping theorem}

\subsection{Bounded support points}

\subsection{Extremal length in higher dimension}}

\section{The Loewner differential equation and the class $S^0_n$.}\label{sec:loewner}

In higher dimensions, a major issue
is the fact that the class
$$ \S_n:=\{f \in \HB \, : \, f(0)=0, Df(0)=\id, f
\text{ univalent}
\}$$
of all normalized univalent mappings on $\B^n$ is not compact for any $n \ge
2$. This is easily seen e.g.~by considering  the noncompact family of shear mappings
$$ z=(z_1,\ldots, z_n)  \mapsto (z_1+\alpha z_2^2,z_2,\ldots, z_n) \, , \qquad
\alpha \in \C \, ,$$
which all belong to $\S_n$. In particular, continuous functionals $ J : \S_n
\to \R$ do not even need to have upper bounds if $n \ge 2$. In order to study extremal
problems for univalent functions in higher dimensions, it is therefore necessary to single out a compact subclass of
$\S_n$, and one of the most studied classes in this connection
is the class $\S^0_n$ of all mappings that admit a so--called parametric
representation by means of the Loewner differential equation. It turns out that the classes $\S^0_n$
are compact for each $n \ge 2$ and that $\S^0_1=\S$.

\medskip

We now briefly describe the classes $\S^0_n$ using almost standard notation.

\begin{definition}
Let $$\Un:=\big\{ h \in \HB \, : \, h(0)=0, Dh(0)=-\id, \, \Re \langle
  h(z),z\rangle \le 0 \text{ for all } z \in \B^n \big\} \, .$$
Here, $\langle \,\cdot,\cdot \rangle$ denotes the canonical Euclidean inner
product of $\C^n$. 
\end{definition}

It is not difficult to see that
a mapping $h \in \HB$ satisfying $h(0)=0$ and $Dh(0)=-\id$ belongs to $\Un$ 
if and only if $\Re \langle -h(z),z \rangle>0$ for all $z \in \B^n \backslash \{0\}$,
see \cite[Remark 2.1]{BHKG2016}. In particular, 
the set $\Un$ is exactly the class $-\mathcal{M}$ with $\M$ as defined e.g.~in
\cite[p.~203]{GK2003}.

\begin{theorem}[\cite{GK2003}: Theorem 6.1.39] \label{thm:classm}
$\Un$ is a compact and convex subset of $\HB$. 
\end{theorem}

\begin{definition}
Let $\R^+_0:=\{t \in \R\, : \, t\ge 0\}$.
 A \textit{Herglotz vector field in the class $\Un$} is a mapping $G : \B^n
 \times \R^+_0 \to \C^n$ such that
\begin{itemize}
\item[(i)] $G(z,\cdot)$ is measurable on $\R^+_0$ for every $z \in \B^n$, and
\item[(ii)] $G(\cdot,t) \in \Un$ for a.e.~$t \in \R^+_0$.
\end{itemize}
\end{definition}

\begin{theorem}[The Loewner Equation] \label{thm:loewner}
Let $G$ be a Herglotz vector field in the class $\Un$.
Then for any $z \in \B^n$ there is a unique solution 
 $\R^+_0 \ni t \mapsto \varphi(t,z) \in \B^n$  of the initial value problem
\begin{equation} \label{eq:L} \begin{array}{rcl}
\displaystyle \frac{\partial \varphi}{\partial t} (t,z)&=& G(\varphi(t,z),t) \quad \text{ for a.e. } t \ge 0 \,, \\[3mm]
\varphi(0,z)&=&z \, .
\end{array}
\end{equation}
For each $t \ge 0$, the mapping $e^t \varphi(t,\cdot) : \B^n \to \C$ belongs to
$\S_n$ and the 
limit
$$f^G:=\lim_{t \to \infty} e^t \varphi(t,\cdot)  $$ exists locally uniformly in
$\B^n$ and belongs to $\S_n$.  
\end{theorem}

We refer to \cite[Thm.~8.1.5]{GK2003} for the proof. The differential equation
in (\ref{eq:L}) is the \textit{Loewner equation} (in $\C^n$). It induces a map
from the set of all Herglotz vector fields in the class $\Un$ into the set $\S_n$.
The range of this map plays a crucial role in this paper:

\begin{theorem} \label{thm:s0}
The set
$$ \mathcal{S}_n^0:=\left\{f^G \in \HB \, | \, G \text{ Herglotz vector
  field  in the class } \Un \right\}  $$
is a compact subset of $\HB$ for each $n\in \N$, and $\S_1^0=\S$.
\end{theorem}

We refer to \cite[Corollary 8.3.11]{GK2003} for a proof of the first
statement. The fundamental fact that $\S_1^0=\S$ is a result of Pommerenke
\cite{Pommerenke1965}, see also \cite[Chapter 6.1]{Pommerenke1975}.
We also note that  $e^t\varphi(t,\cdot) \in \mathcal{S}_n^0$ for all
$t \ge 0$ for every solution to (\ref{eq:L}), see e.g.~\cite[Lemma
2.6]{Schleissinger2014}.
The class $\mathcal{S}_n^0$ is exactly the class of mappings in
$\HB$ which have a \textit{parametric representation} as introduced by Graham,
Hamada and Kohr \cite[Definition 1.5]{GHK2002}, see also 
 \cite{GHK2008,GK2003}.

\medskip

Since the class $\S^0_n$ is compact, we can ask for sharp coefficient bounds
as in the one--dimensional case. More precisely,  let $f \in \S_n^0$. We write $f=(f_1,\ldots, f_n)$ with $f_j \in 
\text{Hol}(\B^n,\C)$, and consider the coefficient functionals 
$$ J_{\alpha}(f):= a_{\alpha} \, , \qquad 
f_1(z)=z_1+\sum \limits_{\alpha \in \N^n_0, |\alpha| \ge 2} a_{\alpha}\,
z^{\alpha}  \, .$$
Here $\alpha=(\alpha_1,\ldots, \alpha_n) \in \N^n_0$ denotes a multi--index
and $|\alpha|=\alpha_1+\ldots+\alpha_n$. In view of the Bieberbach
conjecture, it is natural to consider the extremal problem
\begin{equation} \label{eq:bieberbachgen}
 \max \limits_{f \in \S^0_n} \Re J_{\alpha}(f) 
\end{equation}
for each multi--index $\alpha \in \N_0^N$ such that $|\alpha| \ge 2$.
As in the one--dimensional case, we call a mapping  
$F \in \S^0_n$  extremal for the functional $\Re J_{\alpha}$ over $\S^0_n$, 
if 
$$ \Re J_{\alpha}(F)=\max \limits_{f \in \S^0_n} \Re J_{\alpha}(f) \, .$$

\medskip

We can now formulate the problem.

\begin{problem} \label{prob:main}
For a multi--index $\alpha \in \N_0^N$ with  $|\alpha| \ge 2$
 find a necessary condition for an extremal mapping 
$F \in \S^0_n$  for the functional $\Re J_{\alpha}$ over $\S^0_n$. For the
special case $n=1$ this condition should reduce to the Schiffer differential
equation (\ref{eq:schiffer}). In addition, this necessary condition should be
a sufficient condition for extremality under suitable side conditions.
\end{problem}

\section{Control--theoretic interpretation 
of the Loewner equation}

In this section we give an interpretation of the Loewner equation as an
infinite--dimensional \textit{control system} in the Fr\'echet space $\HB$.
For this purpose, we gradually begin to change notation.

\begin{remhist}
It seems that Loewner himself was the first who came up with the idea of applying
 methods from optimal control theory to the Loewner equation. In 1967, his last student G.S.~Goodman
 \cite{Goodman1967} combined Loewner's theory with the then new Pontryagin
 Maximum Principle. Since then this approach has been used by many others, see
 e.g.~\cite{Popov1969, Alexandrov1976,
   FriedlandSchiffer1976, FriedlandSchiffer1977} and in particular the
important  contributions of D.V.~Prokhorov \cite{Prokhorov1984, Prokhorov1990,
   Prokhorov1993, Prokhorov2002}. Recent applications of optimal control
 methods to univalent functions can be found e.g.~in \cite{ProkhorovSamsonova2015, KochSchleissinger2016}.
\end{remhist}

\begin{definition}[Admissible controls, control set]
Let $G$ be a Herglotz vector field in the class $\Un$. We call
the mapping
$$ u=u_G : \R^+_0 \to \Un \, , \quad t \mapsto G(\cdot,t)\, , $$
an \textit{admissible control} (for the Loewner equation) and 
denote by $\mathcal{U}_n$ the collection of all admissible controls. The class
$\Un$ is called the \textit{control set} (of the Loewner equation).
\end{definition}

\begin{remark}[Herglotz vector fields as $\HB$--valued measurable controls]
One might think of an admissible control as a measurable mapping defined on 
the time interval $\R^+_0=[0,\infty)$ and with values in the control set
$\Un$, which is a subset of the infinite dimensional Fr\'echet space $\HB$.
\end{remark}
\medskip

We continue to change notation and  denote in the sequel the solutions
$\varphi(t,\cdot)$ of the Loewner equation (\ref{eq:L}) by $\varphi_t$.  Note that $\varphi_t$
always belongs to the composition semigroup 
$$ \Sn:=\big\{ \varphi \in \HB \, : \, \varphi(\B^n) \subseteq \B^n, \, \varphi(0)=0\big\} \,$$
of all holomorphic selfmaps of the unit ball $\B^n$.

\begin{definition}[Trajectories, state space]
Let $u \in \mathcal{U}_n$ be an admissible control. Denote by $G$   the 
 Herglotz vector field in the class $\Un$ such that $u=u_G$ and let
 $t \mapsto \varphi_t$ be the solution of the Loewner equation (\ref{eq:L}) corresponding
 to $G$. Then we call the curve 
$$x :=x_u : \R_0^+ \to \Sn \, , \quad t
\mapsto x_u(t):=\varphi_t\, , $$
the \textit{trajectory (of the Loewner equation) for $u$}. 
 The set $\Sn$ is
called the \textit{state space} (of the Loewner equation).
\end{definition}

\begin{remark}[Solutions of the Loewner equation as  $\Sn$--valued curves]
One  might think of an trajectory $x : \R^+_0 \to \Sn$  of the Loewner equation
as an  a.e.~differentiable (and absolutely
continuous) curve defined on the time interval $\R^+_0=[0,\infty)$ and with values in
the state space $\Sn$. The initial point of the curve is the identity
map.
\end{remark}
\medskip

In fact, this remark needs clarification.
If  $x=x_u : \R^+_0 \to \Sn$ is a trajectory of the Loewner equation for $u=u_G$, then
for each $z \in \B^n$ there is set $N_z \subseteq \R^+_0$ of measure zero such
that the solution $t \mapsto \varphi(t,z)$ of (\ref{eq:L}) is differentiable
on $\R^+_0\setminus N_z$ and the Loewner equation holds  for each $t \in \R^+_0\setminus
N_z$. A normal family argument shows that there is in fact a set $N \subseteq
\R^+_0$ of measure zero, which does not depend on $z$, such that the Loewner
equation (\ref{eq:L}) holds for each $t \in \R^+_0\setminus
N$ and each $z \in \B^n$. In addition,
$$ \overset{\bullet}{x}(t):= \frac{\partial \varphi_t}{\partial t}  \in \HB \quad \text{ for
  every } t \in \R^+_0 \setminus N \, .$$
Note that 
$$ \overset{\bullet}{x}(t)= \frac{\partial \varphi}{\partial t}(t,\cdot)=G(\varphi(t,\cdot),t)=u(t)
\circ x(t) \, , \quad t \in \R^+_0\setminus N \, . $$

\begin{conclusion}[The Loewner equation as a control system on $\HB$]
The Loewner equation has the following embarassingly simple form
\begin{equation} \label{eq:L1a} \boxed{\begin{array}{rcl}
\overset{\bullet}{x}&=& u \circ x \, , \\[2mm] x(0)&=&\id \, .
\end{array}}
\end{equation}
This is a control system on the infinite dimensional Fr\'echet
space $\HB$ with state space $\Sn$ and control set $\Un$.
\end{conclusion}

\begin{remhist}
The simple form (\ref{eq:L1a}) of the  Loewner equation is due to
E.~Schippers \cite[Proposition 6]{Schippers2006}.
\end{remhist}

\begin{remark}
If we define $$ g : \Sn \times \Un \to \HB \, , \qquad g(x,u):=u \circ x \, , $$ 
then the Loewner equation (\ref{eq:L1a}) takes the ``traditional'' form of a control
system:
\begin{equation} \label{eq:L2} \begin{array}{rcl}
\overset{\bullet}{x}(t)&=& g(x(t),u(t)) \quad \text{ for a.e. } t \ge 0 \,
, \\[2mm] x(0)&=&\id \, .
\end{array}
\end{equation}
\end{remark}

Note that the Loewner equation (\ref{eq:L1a}) is a control system which is
\textbf{linear with respect to the control}.

\medskip

 We can now consider  the \textit{reachable set for time $t$} of the Loewner
equation, that is, the set $ \{x_u(t) \, : \, u \in \mathcal{U}_n \}$.
However, in view of $e^t x_u(t) \in \S^0_n$ for each $t \ge 0$ and each $u \in
\mathcal{U}_n$ (see Theorem \ref{thm:s0}), we strongly prefer to slightly abuse language
and call
$$ \mathcal{R}(t):=
\left\{ e^t x_u(t) \, : \,  u \in \mathcal{U}_n \right\}
 , \quad t\ \in \R^+_0, \quad \text{ resp. }  \quad \mathcal{R}(\infty):=
\left\{ \lim \limits_{t \to \infty} e^t x_u(t) \, : \,  u \in \mathcal{U}_n \right\}
$$
the reachable set for time $t \in \R^0_+ \cup \{\infty\}$ and 
$$ \mathcal{R}:=\bigcup \limits_{t \in [0,\infty]} \mathcal{R}(t) \,  $$
the \textit{overall reachable set} of the Loewner equation.

\medskip

Using these notions, our considerations can be summarized as follows.

\begin{theorem}[$S^0_n=$ reachable set of the Loewner equation] \label{thm:ll1}
$$ S_n^0=\mathcal{R}=\mathcal{R}(\infty) \, .$$
\end{theorem}

\begin{remhist}
The idea of viewing the classes $\S_n^0$ as reachable sets of
the Loewner equation
has been pioneered by Prokhorov \cite{Prokhorov1990,
  Prokhorov1993, Prokhorov2002} for $n=1$ and by
Graham, Hamada, G.~Kohr and M.~Kohr \cite{GHKK2012,GHKK2014} for $n>1$.

\end{remhist}

With the help of Theorem \ref{thm:ll1},
extremal problems over the class $\S_n^0$ can be treated as
 optimal control problems. 
For the sake of simplicity, we again consider only extremal problems involving Taylor coefficients of univalent maps.

\begin{definition}
Let $\alpha \in \N_0^N$ be a multi--index with $|\alpha| \ge 2$. An admissible
control $u^* \in \mathcal{U}_n$ is called 
 an \textit{optimal control} for the functional $J_{\alpha} : \HB \to \C$
on $\S^0_n$ if the univalent mapping
$$ F_{u^*}:=\lim \limits_{t \to \infty} e^t x_{u^*}(t)  \in \S^0_n \, $$
 is extremal for $\Re J_{\alpha}$ over $\S^0_n$.
\end{definition}

Note carefully, that if $F \in \S^0_n$ is extremal for $\Re J_{\alpha}$ over $\S^0_n$,
then any admissible control $u : [0,\infty) \to \Un$ which generates
$F$ in the sense that $F=F_u$, is an optimal control for $J_{\alpha}$.
Problem \ref{prob:main} can now be stated in control theoretic terms.

\begin{problem} \label{prob:main2}
 Let $\alpha \in \N_0^N$ be a multi--index with $|\alpha| \ge 2$.
Find a necessary condition for an optimal control  for the functional
$J_{\alpha} : \HB \to \C$.   In addition, this necessary condition should be
a sufficient condition for extremality under suitable side conditions.
\end{problem}

\section{The Pontryagin Maximum Principle for the class $\S^0_n$}

The standard necessary condition for an optimal control is provided by 
the 
$$ \fboxsep5mm
\fbox{\textit{Pontryagin Maximum Principle,}} $$
see  \cite{BGP1956}. We first consider the standard finite dimensional version.

\begin{remark}[Prelude to the Pontryagin Maximum Principle]
Let $V$ be a  finite dimensional complex vector space. Suppose that the ``state
space'' $S$ is an open subset of $V$ and that the ``control set'' $U$  is a
compact subset of $V$.  Let $g  :
S \times U \to V$ be a $C^1$--function and let $x_0 \in S$ be a fixed
``initial'' state. 
\begin{itemize}
\item[(I)] \textit{The control system}

Consider 
\begin{equation} \label{eq:control} \begin{array}{rcl}
\overset{\bullet}{x}(t)&=& g(x(t),u(t)) \quad \text{ for a.e. } t \ge 0 \,
, \\[2mm] x(0)&=&x_0 \, ,
\end{array}
\end{equation}
and call a measurable control $u : \R_0^+ \to U$ \textit{admissible} (for the system
(\ref{eq:control})), if
the 
solution $x=x_u$ to (\ref{eq:control}) exists for all a.e.~$t \in \R^+_0$. 

\item[(II)] \textit{The optimal control problem}

Let $J  \in V^*$, that is, let $J : V \to \C$ be a $\C$--linear (continuous) functional and
let $T >0$ be fixed. 
We call an admissible control $u^* : \R_0^+ \to U$ an \textit{optimal control} for the functional
$ J : V \to \C$ on $S$ and time $T>0$, if 
\begin{equation} \label{ex:cc1}
\Re J(x_{u^*}(T)) = \max \big\{ \Re J(x_u(T)) \, : \, u : \R^+_0 \to U \text{ admissible}\big\}.
\end{equation}
\item[(III)] \textit{The adjoint equation}

Suppose that $u : \R_0^+ \to U$ is an admissible control and $x=x_u : \R^+_0 \to
S$ is the solution to (\ref{eq:control}). Then the
 linear (matrix) equation
\begin{equation} \label{eq:controladj} \begin{array}{rcl}
 \overset{\bullet}{\Phi}(t) &=& -\Phi(t) D_x g(x(t),u(t)) \, ,\\[2mm]
\Phi(T) &=& \id\, , 
\end{array}
\end{equation}
is called the \textit{adjoint equation} along $(x,u)$ at  $T$.
Note that $\Phi(t) \in \mathcal{L}(V)$, the set of linear (continuous)
endomorphisms of $V$.

\item[(IV)] \textit{The Hamiltonian}

The function $\mathcal{H} : S \times \mathcal{L}(V) \times U  \to \C$
defined by
$$ \mathcal{H}(x,\Phi,u):=J(\Phi \cdot g(x,u)) $$
is called the \textit{complex Hamiltonian} 
for the extremal problem (\ref{ex:cc1}).
Note that  if we denote the transpose of a map $\Phi \in \mathcal{L}(V)$ by
$\Phi_*$, then 
$$  \mathcal{H}(x,\Phi,u)=\Phi_*(J) g(x,u)$$
\end{itemize}
\end{remark}

We can now state the Pontryagin Maximum Principle:

\begin{theorem}[Pontryagin Maximum Principle] \label{thm:pmpgen}
Let $V$ be a finite dimensional complex vector space and $J \in V^*$. Suppose that $u^* : \R^+_0 \to U$ is an optimal control for $J : V \to \C$ on $S$ and time $T>0$.
Denote by $\Phi^* : \R^+_0 \to \mathcal{L}(V)$  the solution to the adjoint equation
(\ref{eq:controladj}) along $(x_{u^*},u^*)$ at $T$, then
$$ \Re \,\mathcal{H}(x_{u^*}(t),\Phi^*(t),u^*(t))=\max \limits_{u \in U} \Re\,
\mathcal{H}(x_{u^*}(t),\Phi^*(t),u)\quad \text{ for a.e. } t \in [0,T]\, .$$
\end{theorem}

The conclusion of Theorem \ref{thm:pmpgen}  means that 
for a.e.~$t \in \R^+_0$ the value $u^*(t)$ of the optimal
control $u^*$ provides a maximum for  the function
$$ u \mapsto \Re \Phi^*(t)_*(J) g(x_*(t),u) \, 
$$
over the control set $U$. For a proof of Theorem \ref{thm:pmpgen} we refer to
any textbook on optimal control theory, see e.g.~\cite[p.~152 ff.]{Zabczyk1992}.

\begin{remark}[On the definition of the Hamiltonian/The costate equation]
Our definition of the Hamiltonian is slightly nonstandard.
However, it is easy to see its relation to the standard Hamiltonian formalism.
Note that for $V=\C^n$ we can identify the dual space $V^*$ with $V$ and 
we can write the $\C$--linear functional $J : V \to\C$
as $J(y)=\overline{\eta}^Ty$ for some vector $\eta
\in V$. Hence, if $\Phi : \R^+_0 \to \mathcal{L}(V)$ is the solution of the adjoint
equation (\ref{eq:controladj}) along $(x,u)$ at $T$, then
 $$\Psi:=\overline{\eta}^T \Phi=J(\Phi \, \cdot )=\Phi_*(J)  : \R^0_+ \to V^*$$
 is the solution to the so--called
\textit{costate equation}
\begin{equation} \label{eq:controladj2} \begin{array}{rcl}
 \overset{\bullet}{\Psi}(t) &=& -\Psi(t) D_x g(x(t),u(t)) \, ,\\[2mm]
\Psi(T) &=& \overline{\eta}^T\, . 
\end{array}
\end{equation}
Therefore, the  \textit{standard} complex Hamiltonian for the extremal problem (\ref{ex:cc1}),
\begin{equation} \label{eq:hamiltonstandard}
 H : S \times V^* \times U  \to \C \, , \quad 
H(x,\Psi,u):=\Psi^T g(x,u)\, ,
\end{equation} 
see \cite{Jurdjevic1997,Prokhorov1993},  is related to (our) complex Hamiltonian by
$$\mathcal{H}(x,\Phi,u)=H(x,\Psi,u) \, , \qquad \Psi=\overline{\eta}^T \Phi\,
.$$
\end{remark}

We next apply the Pontryagin machinery as outlined above to the abstract Loewner equation
(\ref{eq:L1a}), so we replace the finite dimensional complex vector space $V$, the
state space $S\subseteq V$, the control set $U \subseteq V$ and the control
system $g : S \times U \to V$ 
by
$$ V=\HB \, , \quad S=\Sn \, , \quad \, U=\Un \, , \quad
g(x,u)=x \circ u \, .$$

\renewcommand{\arraystretch}{2}
\begin{table}[h]
\begin{center}
\begin{tabular}{|l||c||c|} \hline
& Control System on $\C^n$ & Loewner Equation \\ \hline \hline
State space & $S \subseteq \C^n$ & $\Sn \subseteq \HB$ \\ \hline
Control set & $U \subseteq \C^n$ & $\Un \subseteq \HB$ \\ \hline
Controls    & $ u : \R^+_0 \to U$ & $u : \R^+_0 \to \Un$ \\ \hline
Trajectories & $x : \R^+_0 \to S$ & $x : \R^+_0 \to \Sn$ \\ \hline
Equation & $\overset{\bullet}{x}=g(x,u)$ & $\overset{\bullet}{x}=u \circ x$\\ \hline
\end{tabular}
\end{center}
\caption{Loewner equation as a control system on $\HB$}
\end{table}

\begin{remark}
As we shall see, our sligthly nonstandard definition of the Hamiltonian proves itself
as user--friendly. The main reason for this is that the dual $\HB^*$
of $\HB$ can no longer be identified with $\HB$, and in fact the topological
structure of  $\HB^*$ is fairly complicated, see
\cite{Caccioppoli1931,Toeplitz1949, Grothendieck1953a, Grothendieck1953b}. However, it will turn out that in case of
the Loewner equation, the solutions of the associated adjoint equation actually live in the
``decent'' vector space $ \text{Hol}(\B^n,\C^{n \times n}) $
of all holomorphic maps from $\B^n$ into $\C^{n \times n}$. Note that each $A
\in \text{Hol}(\B^n,\C^{n \times n})$ gives rise to the continuous linear
operator $\mathcal{A} \in \mathcal{L}(\HB)$ defined by 
$$ \mathcal{A} f:=A(\cdot) f \, ,\qquad f \in \HB \, .$$
\end{remark}

 We proceed in a purely formal way in order to emphasize the
analogy with the finite dimensional case, but wish to point out that 
the following formal considerations can be made rigorous using the elegant Fr\'echet
space calculus developed by R.~Hamilton \cite{Hamilton1982}.
We start with the adjoint equation for the Loewner ODE.

\begin{remark}[The adjoint equation of the Loewner ODE]
Recall the Loewner ODE in abstract form
$$ \overset{\bullet}{x}=u \circ x \, .$$
Taking the ``functional'' derivative of the mapping on the right-hand side,
$$ g : \Sn \times \Un \to \HB \, , \qquad g(x,u):=u \circ x \, , $$ 
with respect to $x \in \Sn$, 
we get 
$$ \frac{\partial (u \circ x)}{\partial x}=Du \circ x \, .$$
Here, $Du \in \text{Hol}(\B^n,\C^{n \times n})$ is the total derivative of $u
\in  \HB$ and hence $Du \circ x \in \HB$. Therefore,
 the adjoint equation of the Loewner ODE is
\begin{equation} \label{eq:controladj3} \begin{array}{rcl}
 \overset{\bullet}{\Phi}(t)&=&-\Phi(t) \cdot Du(t) \circ x(t) \, , \\[-2mm]
\Phi(T) &=& \id\, . 
\end{array}
\end{equation}
Recall that for each $t \in \R^+_0$ the map 
$\Phi(t)$ is a continuous linear operator on $\HB$.
It is now easy to find the solution of (\ref{eq:controladj3}) in terms of $Dx$.
In fact, just take the total derivative of
$\overset{\bullet}{x}=u \circ x \in \HB$ and get
$$ \overset{\bullet}{Dx}=Du \circ x \cdot Dx \, . $$
Hence using $Dx \cdot (Dx)^{-1}=\id$, we see that
$$ \overset{\bullet}{(Dx)^{-1}}=-(Dx)^{-1} Du \circ x \, .$$
This means that the solution to the adjoint equation (\ref{eq:controladj3}) 
along $(x,u)$ at $T$ is simply
$$ t\mapsto \Phi(t)=D x(T) \cdot (Dx(t))^{-1} \in \text{Hol}(\B^n,\C^{n \times n})
\subseteq \mathcal{L}(\HB) \,  , $$
and the corresponding Hamiltonian of the Loewner equation (\ref{eq:L1a}) for
the extremal problem (\ref{ex:cc1})  has the form
$$ \mathcal{H}(x,\Phi,u)=J\left(D x(T) \cdot (Dx)^{-1} \cdot u \circ
  x\right)\, , \quad u \in \Un \, .$$
\end{remark}

\begin{definition}
Let $J\in \HB^*$ and $T>0$. Suppose that $u^* \in \mathcal{U}_n$. Then we
define for each $t \ge 0$,
$$L_t \in  \HB^* \, , \qquad L_t(u):=J\left(D x_{u^*}(T) \cdot \left(Dx_{u^*}(t)\right)^{-1} \cdot u \circ
  x_{u^*}(t)\right)\, .$$
\end{definition}

Now we can state the analogue of Theorem \ref{thm:pmpgen} for the Loewner
equation and general linear functionals $J \in \HB^*$.

\begin{theorem}[Pontryagin Maximum Principle for the Loewner equation]
Let $J\in \HB^*$.
Suppose that $u^* : \R^+_0 \to \Un$ is an optimal control for
$J : \HB \to \C$ on $\Sn$ and time $T>0$.
Then
$$ \Re \,L_t(u^*(t)) =\max \limits_{u \in\Un} \Re\,L_t(u) 
\quad \text{ for a.e. } t \in [0,T]\, .$$
\end{theorem}

It is now a short step to a statement of the Pontryagin Maximum Principle
for the class $\S^0_n$ and coefficient functionals $J_{\alpha}$.

\begin{theorem}[Pontryagin Maximum Principle for the class $\S^0_n$] \label{thm:main2a}
Let $\alpha \in \N_0^N$ with $|\alpha| \ge 2$ and let $F \in \mathcal{S}_n^0$ be an extremal mapping
for $\Re J_{\alpha}$ over $\S^0_n$. 
 Suppose that $G$ is a Herglotz vector field in the
class $\Un$ such that $F=f^G$. Denote by $\varphi_t: \R^+_0 \to \Sn$ the solution to
the Loewner equation (\ref{eq:L}). 
 For each $t \ge 0$ define 
\begin{equation} \label{eq:hamiltonian}
 L_t \in \HB^* , \quad 
h \mapsto L_t(h):=J_{\alpha} \left( DF \cdot \left[  D\varphi_t  \right]^{-1} \cdot h\left(\varphi_t\right) \right) \, . 
\end{equation}
Then for a.e.~$t \ge 0$,
\begin{equation} \label{eq:pontryagin1}
\Re L_t(G(\cdot, t))=\max \limits_{h \in \Un}  \Re L_t(h) \, .
\end{equation}
\end{theorem}

See \cite{Roth2015} for a rigorous proof of Theorem \ref{thm:main2a}.

\begin{conclusion}
Let $\alpha \in \N_0^N$ with $|\alpha| \ge 2$ and let $F \in \mathcal{S}_n^0$.
Suppose that $G$ is a Herglotz vector field in the
class $\Un$ such that $F=f^G$.
 Then the condition that 
\begin{center}  \fboxsep5mm
\fbox{ $G(\cdot, t)$ maximizes $\Re L_t$ over $\Un$ for a.e.~$t \ge 0$}
\end{center}
is necessary for $F$ being extremal for $\Re J_{\alpha}$ over $\S^0_n$.
\end{conclusion}

\section{The Schiffer differential equation and Pontryagin's maximum principle}

We now briefly describe the intimate relation between Schiffer's differential
equation and Pontryagin's Maximum Principle for the case $\S^0_1=\S$.
For the sake of simplicity, we restrict again to the case of the $N$--th coefficient functional $J_N :\S \to
\C$, $J_N(f):=f^{(N)}(0)/N!$.

\medskip

\begin{theorem} \label{thm:main123}
Let  $F \in \S$ and suppose that $G$ is a Herglotz vector field in the
class $\mathbb{U}_1$ such that $F=f^G$. 
Denote by $\varphi_t: \R^+_0 \to \mathbb{S}_1$ the solution to
the Loewner equation (\ref{eq:L}). 
 For each $t \ge 0$ define 
\begin{equation} \label{eq:hamiltoniannew}
 L_t \in \textrm{\rm Hol}(\D,\C) , \quad 
h \mapsto L_t(h):=J_N \left( F' \cdot \left( \varphi'_t  \right)^{-1} \cdot h\left(\varphi_t\right) \right) \, . 
\end{equation}
Then the following conditions are equivalent.
\begin{itemize}
\item[(a)] $F$ ist a solution of the Schiffer differential equation
  (\ref{eq:schiffer}), that is,
\begin{equation} \label{eq:schiffernew}
 \left[ \frac{z F'(z)}{F(z)} \right]^2 P_N\left(\frac{1}{F(z)}\right)=R_N(z)
  \, ,
\end{equation}
with 
$P_N$ resp. $R_N$ defined by (\ref{eq:Pn}) resp.~(\ref{eq:Rn}), and the positivity condition (\ref{eq:schifferboundary}) holds.

\item[(b)] For a.e.~$t \ge 0$,
\begin{equation} \label{eq:pontryaginnew}
\max \limits_{h \in \mathbb{U}_1}  \Re L_t(h) = \Re L_t(G(\cdot,t)) \, .
\end{equation}
\end{itemize}
\end{theorem}

\begin{remark}[Schiffer's equation = Pontryagin's Maximum Principle]
Note that \textit{under the assumption that $F$ is extremal for the functional
  $J_N$ over $\S$ we have:}
\begin{itemize}
\item[(i)] Condition (a) = Conclusion of Schiffer's Theorem
  \ref{thm:schiffer};
\item[(ii)] Condition (b) = Conclusion of Pontryagin's Maximum Principle\\
   \hspace*{3cm} (Theorem \ref{thm:main2a} for the  case $n=1$).
\end{itemize}
Hence Theorem \ref{thm:main123} says that the two necessary conditions for $F$
being extremal for the functional $J_N$ over $\S$ provided by Schiffer's
theorem (Condition (a)) and by Pontryagin's Maximum Principle (Condition (b))
are in fact \textit{equivalent}.
\end{remark}

A few more words are in order.

\begin{remark}
The implication (a) $\Longrightarrow$ (b) of Theorem \ref{thm:main123}
can in fact be traced back to the work of Schaeffer, Schiffer and Spencer, see
\cite{Schiffer1945,SSS1949,SchaefferSpencer1950}. Consequently, a version of
Pontryagin's Maximum Principle for a certain control system, namely the  Loewner equation 
 for $n=1$, has been known more than a decade before the discovery of
 the maximum principle by Pontryagin and his coauthors in 1956.
In fact, Leung \cite{Leung1985} speaks in this regard of the
\textit{Schiffer--Pontryagin Maximum Principle}. The converse implication 
(b) $\Longrightarrow$ (a) seems to lie deeper, see \cite{Goodman1967, Popov1969,
Roth2000}.

\hide{We give a quick outline of
the proof. Suppose that $F=f^G \in \S$ for some Herglotz vector field $G$ in the
class $\mathbb{U}_1$. Since $F$ is a solution of the Schiffer differential
equation, it is a one--slit map and it is well--known (see e.g.~\cite{Pommerenke1975,Duren1983})
that there is a measurable (even continuous) function $\kappa : \R^+_0 \to \partial \D$ such
that 
$$ G(z,t)=-z \frac{\kappa(t)+z}{\kappa(t)-z} \, .$$
Let $F(z,t)$ denote the corresponding normal Loewner
chain which is a normalized solution to the Loewner PDE
$$ \frac{\partial F}{\partial t}(z,t)=-F'(z,t)
G(z,t) \, .$$
Here $'=\frac{\partial}{\partial z}$.
Consider  
$$ Q(z,t):=\left[ \frac{z F'(z,t)}{F(z,t) }\ \right]^2 P_n(1/F(z,t)) \, ,$$
and note that $Q(\xi,t) \ge 0$ for each $\xi \in \partial \D$ with equality
for $\xi=\kappa(t)$ since $\kappa(t)$ is mapped by $F(\cdot, t)$ onto the tip
of the (erased) slit. This means that we have
$$ Q(\eta,t)-Q(\kappa(t),t) \ge 0$$}
\end{remark}

\begin{conclusion}[Extending Schiffer's theorem to higher dimensions]
Theorem \ref{thm:main2a} generalizes Schiffer's Theorem \ref{thm:schiffer}
from the class $\S=\S^0_1$ to any of the classes $\S^0_n$.
\end{conclusion}

\section{Teichm\"uller's coefficient theorem without quadratic differentials}

Using Theorem \ref{thm:main123} it is now easy to give a formulation of
Teichm\"uller's Coefficient Theorem solely in terms of the Loewner differential
equation.

\begin{theorem} \label{thm:teichneu}
Let  $F \in \S$ and suppose that $G$ is a Herglotz vector field in the
class $\mathbb{U}_1$ such that $F=f^G$. 
Denote by $\varphi_t: \R^+_0 \to \mathbb{S}_1$ the solution to
the Loewner equation (\ref{eq:L}). 
 For each $t \ge 0$ define $L_t \in \textrm{\rm Hol}(\D,\C)$ by
 (\ref{eq:hamiltoniannew}).
If  for a.e.~$t \ge 0$,
\begin{equation} \label{eq:pontryaginnew2}
\max \limits_{h \in \mathbb{U}_1}  \Re L_t(h) = \Re L_t(G(\cdot,t)) \, ,
\end{equation}
then 
$$ \Re J_N(f) \le \Re J_N(F)$$
for any  
$$f \in \mathcal{S}(A_2,\ldots, A_{N-1})\, .$$
\end{theorem}

\begin{proof}
By Theorem \ref{thm:main123} and Remark \ref{rem:schiffer} (d) 
this exactly is Teichm\"uller's Coefficient
Theorem \ref{thm:teich}. \end{proof}

Note that  in view of  Pontryagin's Maximum Principle,
condition (\ref{eq:pontryaginnew2}) is necessary
for $F \in \S$ being an extremal function for $J_N$ over $\S$.
Theorem \ref{thm:teichneu} simply says that this  condition is also
sufficient for extremality under a suitable side condition.

\begin{conclusion}
Let $F(z)=z+\sum_{k=2}^{\infty} A_k z^k \in \mathcal{S}$ and suppose that $G$ is a Herglotz vector field in the
class $\mathbb{U}_1$ such that $F=f^G$. For each $t \ge 0$ define $L_t \in \textrm{\rm Hol}(\D,\C)$ by
 (\ref{eq:hamiltoniannew}).
 Then the condition that 
\begin{center}  \fboxsep5mm
\fbox{ $G(\cdot, t)$ maximizes $\Re L_t$ over $\mathbb{U}_1$ for a.e.~$t \ge 0$}
\end{center}
is 
\begin{itemize}
\item[(a)]  necessary for $F$ being extremal for $\Re J_{N}$ over $\S$, and 
\item[(b)] 
sufficient for $F$ being extremal for $\Re J_{N}$ over $\S(A_2,\ldots, A_{N
-1})$.
\end{itemize}
\end{conclusion}

\begin{problem} \label{prob:final}
Find a proof of Theorem \ref{thm:teichneu} using only the Loewner differential
equation. We note that the  standard method in control theory for obtaining sufficient
conditions for optimal control functions makes use of Bellman functions.  See
\cite{Prokhorov1993} for some application of Bellman functions to the Loewner equation.
\end{problem}

\begin{problem}
Let  $\alpha \in \N^N_0$ be a multi--index with $|\alpha|\ge 2$, 
$F \in \S^0_n$ and suppose that $G$ is a Herglotz vector field in the
class $\mathbb{U}_n$ such that $F=f^G$. 
Denote by $\varphi_t: \R^+_0 \to \mathbb{S}_n$ the solution to
the Loewner equation (\ref{eq:L}). 
 For each $t \ge 0$ define $L_t \in \HB$ by
 (\ref{eq:hamiltonian}). Assume that
  for a.e.~$t \ge 0$,
\begin{equation} \label{eq:pontryaginnew3}
\max \limits_{h \in \mathbb{U}_n}  \Re L_t(h) = \Re L_t(G(\cdot,t)) \, .
\end{equation}
Is there a subset $\S_{\alpha}$ of $\S^0_n$ such that 
$$ \Re J_{\alpha}(f) \le \Re J_{\alpha}(F)$$
for any  $f \in \S_{\alpha}$\,?
In other words, is the necessary condition for $F \in
\S^0_n$ being an extremal function for $J_{\alpha}$ over $\S^0_n$ provided by
Pontryagin's Maximum Principle also sufficient under a
suitable side condition\,? For $n=1$ the answer is ``Yes'' by Theorem
\ref{thm:teichneu}, which we have seen is equivalent to Teichm\"uller's
Coefficient Theorem \ref{thm:teich}.
 An affirmative answer for $n>1$ would therefore provide an extension of
Teichm\"uller's Coefficient Theorem to higher dimensions.
\end{problem}


\bibliography{UnivalentFunctions}

\begin{thebibliography}{10}

\bibitem{Alexandrov1976}
I.~{Aleksandrov}.
\newblock {Parametric continuations in the theory of univalent functions.
  (Parametricheskie prodolzhenya teorii odnolistnykh funktsij).}
\newblock {Moskva: ''Nauka''. 343 p.}, 1976.

\bibitem{BGP1956}
V.~G. Boltyanski\u\i, R.~V. Gamkrelidze, and L.~S. Pontryagin.
\newblock On the theory of optimal processes.
\newblock {\em Dokl. Akad. Nauk SSSR (N.S.)}, 110:7--10, 1956.

\bibitem{BCTV2014}
F.~{Bracci}, M.~D. {Contreras}, S.~{D{\'i}az-Madrigal}, and A.~{Vasil'ev}.
\newblock {Classical and stochastic L{\"o}wner-Kufarev equations.}
\newblock In {\em {Harmonic and complex analysis and its applications}}, pages
  39--134. Cham: Birkh{\"a}user/Springer, 2014.

\bibitem{BHKG2016}
F.~Bracci, I.~Graham, H.~Hamada, and G.~Kohr.
\newblock Variation of {L}oewner chains, extreme and support points in the
  class {$S^0$} in higher dimensions.
\newblock {\em Constr. Approx.}, 43(2):231--251, 2016.

\bibitem{BracciRoth2017}
F.~Bracci and O.~Roth.
\newblock Support points and the bieberbach conjecture in higher dimension.

\bibitem{Caccioppoli1931}
R.~{Caccioppoli}.
\newblock {Sui funzionali lineari nel campo delle funzioni analitiche.}
\newblock {\em Atti Accad. Naz. Lincei, Rend., VI. Ser.}, 13:263--266, 1931.

\bibitem{deBranges1985}
L.~{de Branges}.
\newblock {A proof of the Bieberbach conjecture.}
\newblock {\em {Acta Math.}}, 154:137--152, 1985.

\bibitem{Duren1983}
P.~L. Duren.
\newblock {\em Univalent functions}, volume 259 of {\em Grundlehren der
  Mathematischen Wissenschaften [Fundamental Principles of Mathematical
  Sciences]}.
\newblock Springer-Verlag, New York, 1983.

\bibitem{FriedlandSchiffer1976}
S.~Friedland and M.~Schiffer.
\newblock Global results in control theory with applications to univalent
  functions.
\newblock {\em Bull. Amer. Math. Soc.}, 82(6):913--915, 1976.

\bibitem{FriedlandSchiffer1977}
S.~Friedland and M.~Schiffer.
\newblock On coefficient regions of univalent functions.
\newblock {\em J. Analyse Math.}, 31:125--168, 1977.

\bibitem{Goodman1967}
G.~S. Goodman.
\newblock {\em U{NIVALENT} {FUNCTIONS} {AND} {OPTIMAL} {CONTROL}}.
\newblock ProQuest LLC, Ann Arbor, MI, 1967.
\newblock Thesis (Ph.D.)--Stanford University.

\bibitem{GHK2002}
I.~Graham, H.~Hamada, and G.~Kohr.
\newblock Parametric representation of univalent mappings in several complex
  variables.
\newblock {\em Canad. J. Math.}, 54(2):324--351, 2002.

\bibitem{GHK2008}
I.~Graham, H.~Hamada, G.~Kohr, and M.~Kohr.
\newblock Parametric representation and asymptotic starlikeness in {$\Bbb
  C^n$}.
\newblock {\em Proc. Amer. Math. Soc.}, 136(11):3963--3973, 2008.

\bibitem{GHKK2012}
I.~Graham, H.~Hamada, G.~Kohr, and M.~Kohr.
\newblock Extreme points, support points and the {L}oewner variation in several
  complex variables.
\newblock {\em Sci. China Math.}, 55(7):1353--1366, 2012.

\bibitem{GHKK2014}
I.~Graham, H.~Hamada, G.~Kohr, and M.~Kohr.
\newblock Extremal properties associated with univalent subordination chains in
  {$\Bbb{C}^n$}.
\newblock {\em Math. Ann.}, 359(1-2):61--99, 2014.

\bibitem{GK2003}
I.~Graham and G.~Kohr.
\newblock {\em Geometric function theory in one and higher dimensions}, volume
  255 of {\em Monographs and Textbooks in Pure and Applied Mathematics}.
\newblock Marcel Dekker, Inc., New York, 2003.

\bibitem{Grothendieck1953a}
A.~Grothendieck.
\newblock Sur certains espaces de fonctions holomorphes. i.
\newblock {\em J. Reine Angew. Math.}, 192:35--64, 1953.

\bibitem{Grothendieck1953b}
A.~Grothendieck.
\newblock Sur certains espaces de fonctions holomorphes. ii.
\newblock {\em J. Reine Angew. Math.}, 192:77--95, 1953.

\bibitem{Hamilton1982}
R.~S. Hamilton.
\newblock The inverse function theorem of {N}ash and {M}oser.
\newblock {\em Bull. Amer. Math. Soc. (N.S.)}, 7(1):65--222, 1982.

\bibitem{Jenkins1958}
J.~A. Jenkins.
\newblock {\em Univalent functions and conformal mapping}.
\newblock Ergebnisse der Mathematik und ihrer Grenzgebiete. Neue Folge, Heft
  18. Reihe: Moderne Funktionentheorie. Springer-Verlag,
  Berlin-G{\"o}ttingen-Heidelberg, 1958.

\bibitem{Jurdjevic1997}
V.~Jurdjevic.
\newblock {\em Geometric control theory}, volume~52 of {\em Cambridge Studies
  in Advanced Mathematics}.
\newblock Cambridge University Press, Cambridge, 1997.

\bibitem{KochSchleissinger2016}
J.~{Koch} and S.~{Schlei{\ss}inger}.
\newblock {Value ranges of univalent self-mappings of the unit disc.}
\newblock {\em {J. Math. Anal. Appl.}}, 433(2):1772--1789, 2016.

\bibitem{Leung1985}
Y.~J. Leung.
\newblock Notes on {L}oewner differential equations.
\newblock In {\em Topics in complex analysis ({F}airfield, {C}onn., 1983)},
  volume~38 of {\em Contemp. Math.}, pages 1--11. Amer. Math. Soc., Providence,
  RI, 1985.

\bibitem{Pommerenke1965}
C.~Pommerenke.
\newblock {\"U}ber die {S}ubordination analytischer {F}unktionen.
\newblock {\em J. Reine Angew. Math.}, 218:159--173, 1965.

\bibitem{Pommerenke1975}
C.~Pommerenke.
\newblock {\em Univalent functions}.
\newblock Vandenhoeck \& Ruprecht, G{\"o}ttingen, 1975.
\newblock With a chapter on quadratic differentials by Gerd Jensen, Studia
  Mathematica/Mathematische Lehrb{\"u}cher, Band XXV.

\bibitem{Popov1969}
V.~{Popov}.
\newblock {L. S. Pontryagin's maximum principle in the theory of univalent
  functions.}
\newblock {\em {Sov. Math., Dokl.}}, 10:1161--1164, 1969.

\bibitem{Prokhorov1984}
D.~{Prokhorov}.
\newblock {The method of optimal control in an extremal problem on a class of
  univalent functions.}
\newblock {\em {Sov. Math., Dokl.}}, 29:301--303, 1984.

\bibitem{ProkhorovSamsonova2015}
D.~{Prokhorov} and K.~{Samsonova}.
\newblock {Value range of solutions to the chordal Loewner equation.}
\newblock {\em {J. Math. Anal. Appl.}}, 428(2):910--919, 2015.

\bibitem{Prokhorov1990}
D.~V. Prokhorov.
\newblock Sets of values of systems of functionals in classes of univalent
  functions.
\newblock {\em Mat. Sb.}, 181(12):1659--1677, 1990.

\bibitem{Prokhorov1993}
D.~V. Prokhorov.
\newblock {\em Reachable set methods in extremal problems for univalent
  functions}.
\newblock Saratov University Publishing House, Saratov, 1993.

\bibitem{Prokhorov2002}
D.~V. Prokhorov.
\newblock Bounded univalent functions.
\newblock In {\em Handbook of complex analysis: geometric function theory,
  {V}ol.\ 1}, pages 207--228. North-Holland, Amsterdam, 2002.

\bibitem{Roth2000}
O.~Roth.
\newblock Pontryagin's maximum principle in geometric function theory.
\newblock {\em Complex Variables Theory Appl.}, 41(4):391--426, 2000.

\bibitem{Roth2015}
O.~Roth.
\newblock Pontryagin's maximum principle for the {L}oewner equation in higher
  dimensions.
\newblock {\em Canad. J. Math.}, 67(4):942--960, 2015.

\bibitem{SSS1949}
A.~C. Schaeffer, M.~Schiffer, and D.~C. Spencer.
\newblock The coefficient regions of schlicht functions.
\newblock {\em Duke Math. J.}, 16:493--527, 1949.

\bibitem{SchaefferSpencer1950}
A.~C. Schaeffer and D.~C. Spencer.
\newblock {\em Coefficient {R}egions for {S}chlicht {F}unctions}.
\newblock American Mathematical Society Colloquium Publications, Vol. 35.
  American Mathematical Society, New York, N. Y., 1950.
\newblock With a Chapter on the Region of the Derivative of a Schlicht Function
  by Arthur Grad.

\bibitem{Schiffer1938}
M.~{Schiffer}.
\newblock {A method of variation within the family of simple functions.}
\newblock {\em {Proc. Lond. Math. Soc. (2)}}, 44:432--449, 1938.

\bibitem{Schiffer1945}
M.~Schiffer.
\newblock Sur l'{\'e}quation diff{\'e}rentielle de {M}. {L}{\"o}wner.
\newblock {\em C. R. Acad. Sci. Paris}, 221:369--371, 1945.

\bibitem{Schippers2006}
E.~Schippers.
\newblock The power matrix, coadjoint action and quadratic differentials.
\newblock {\em J. Anal. Math.}, 98:249--277, 2006.

\bibitem{Schleissinger2014}
S.~Schleissinger.
\newblock On support points of the class {$S^0(B^n)$}.
\newblock {\em Proc. Amer. Math. Soc.}, 142(11):3881--3887, 2014.

\bibitem{Strebel1984}
K.~{Strebel}.
\newblock {Quadratic differentials.}
\newblock {Ergebnisse der Mathematik und ihrer Grenzgebiete. 3. Folge, Band 5.
  Berlin etc.: Springer-Verlag. XII, 184 p.}, 1984.

\bibitem{Teichmueller1938}
O.~{Teichm{\"u}ller}.
\newblock {Ungleichungen zwischen den Koeffizienten schlichter Funktionen.}
\newblock {\em {Sitzungsber. Preu{\ss}. Akad. Wiss., Phys.-Math. Kl.}},
  1938:363--375, 1938.

\bibitem{Toeplitz1949}
O.~Toeplitz.
\newblock Die linearen vollkommenen {R}{\"a}ume der {F}unktionentheorie.
\newblock {\em Comment. Math. Helv.}, 23:222--242, 1949.

\bibitem{Zabczyk1992}
J.~Zabczyk.
\newblock {\em Mathematical control theory: an introduction}.
\newblock Systems \& Control: Foundations \& Applications. Birkh{\"a}user
  Boston, Inc., Boston, MA, 1992.

\end{thebibliography}

\bibliographystyle{plain}

\vfill

Oliver Roth\\
Department of Mathematics\\
University of W\"urzburg\\
Emil Fischer Stra{\ss}e 40\\
97074 W\"urzburg\\
Germany\\
roth@mathematik.uni-wuerzburg.de

\end{document}